\documentclass[english]{article}
\usepackage[T1]{fontenc}
\usepackage[latin9]{inputenc}
\usepackage{amsmath}
\usepackage{amsthm}
\usepackage{amssymb}
\usepackage{graphicx}

\makeatletter
\theoremstyle{plain}
\newtheorem{thm}{\protect\theoremname}
\theoremstyle{definition}
\newtheorem{problem}[thm]{\protect\problemname}
\theoremstyle{remark}
\newtheorem{rem}[thm]{\protect\remarkname}
\theoremstyle{plain}
\newtheorem{prop}[thm]{\protect\propositionname}
\theoremstyle{definition}
\newtheorem{defn}[thm]{\protect\definitionname}

\usepackage{fullpage}
\usepackage{hyperref}
\hypersetup{
  colorlinks   = true, 
  urlcolor     = red, 
  linkcolor    = blue, 
  citecolor   = blue 
}
\date{}

\makeatother

\usepackage{babel}
\providecommand{\definitionname}{Definition}
\providecommand{\problemname}{Problem}
\providecommand{\propositionname}{Proposition}
\providecommand{\remarkname}{Remark}
\providecommand{\theoremname}{Theorem}

\begin{document}
\global\long\def\goinf{\rightarrow\infty}%
\global\long\def\gozero{\rightarrow0}%
\global\long\def\bra{\langle}%
\global\long\def\ket{\rangle}%
\global\long\def\union{\cup}%
\global\long\def\intersect{\cap}%
\global\long\def\abs#1{\left|#1\right|}%
\global\long\def\norm#1{\left\Vert #1\right\Vert }%
\global\long\def\floor#1{\left\lfloor #1\right\rfloor }%
\global\long\def\ceil#1{\left\lceil #1\right\rceil }%
\global\long\def\expect{\mathbb{E}}%
\global\long\def\e{\mathbb{E}}%
\global\long\def\r{\mathbb{R}}%
\global\long\def\n{\mathbb{N}}%
\global\long\def\q{\mathbb{Q}}%
\global\long\def\c{\mathbb{C}}%
\global\long\def\z{\mathbb{Z}}%
\global\long\def\grad{\nabla}%
\global\long\def\t{\mathbb{T}}%
\global\long\def\all{\forall}%
\global\long\def\eps{\varepsilon}%
\global\long\def\quadvar#1{V_{2}^{\pi}\left(#1\right)}%
\global\long\def\cal#1{\mathcal{#1}}%
\global\long\def\cross{\times}%
\global\long\def\del{\nabla}%
\global\long\def\parx#1{\frac{\partial#1}{\partial x}}%
\global\long\def\pary#1{\frac{\partial#1}{\partial y}}%
\global\long\def\parz#1{\frac{\partial#1}{\partial z}}%
\global\long\def\part#1{\frac{\partial#1}{\partial t}}%
\global\long\def\partheta#1{\frac{\partial#1}{\partial\theta}}%
\global\long\def\parr#1{\frac{\partial#1}{\partial r}}%
\global\long\def\curl{\nabla\times}%
\global\long\def\rotor{\nabla\times}%
\global\long\def\one{\mathbf{1}}%
\global\long\def\Hom{\text{Hom}}%
\global\long\def\p{\mathbb{P}}%
\global\long\def\almost{\mathbf{\approx}}%
\global\long\def\tr{\text{Tr}}%
\global\long\def\var{\text{Var}}%
\global\long\def\cov{\text{Cov}}%
\global\long\def\onenorm#1{\left\Vert #1\right\Vert _{1}}%
\global\long\def\twonorm#1{\left\Vert #1\right\Vert _{2}}%
\global\long\def\Inj{\mathfrak{Inj}}%
\global\long\def\inj{\mathsf{inj}}%
\global\long\def\Inf{\mathrm{Inf}}%
\global\long\def\i{\mathrm{I}}%
\global\long\def\sign{\mathrm{sign}}%
\global\long\def\tensor{\otimes}%
\global\long\def\axis{\Delta_{\dagger}}%
\global\long\def\tribes{\mathrm{Tribes}}%

\global\long\def\g{\mathfrak{\cal G}}%
\global\long\def\f{\mathfrak{\cal F}}%
\global\long\def\nonnegative{\left[0,\infty\right)}%
\global\long\def\index#1#2{n\left(#1,#2\right)}%
\global\long\def\hypergeometric#1#2#3#4{_{2}F_{1}\left(#1,#2;#3;#4\right)}%

\title{The distance problem via subadditivity}
\author{Renan Gross\thanks{University of Cambridge. Email: rg751@cam.ac.uk}}
\maketitle
\begin{abstract}
In a recent paper, Aldous, Blanc and Curien asked which distributions
can be expressed as the distance between two independent random variables
on some separable measured metric space. We show that every nonnegative
discrete distribution whose support contains $0$ arises in this way,
as well as a class of finitely supported distributions with density.
\end{abstract}

\section{Introduction}

\subsection{The problem}

The \emph{distance problem} asks the following. Given a distribution
$\theta$ on $\nonnegative$, is it possible to find a complete separable
metric space $\left(S,d\right)$ and a Borel probability measure $\mu$
on $S$, so that when $X,Y\sim\mu$ are independent, their distance
satisfies $d\left(X,Y\right)\sim\theta$? If so, we say that $\left(S,d,\mu\right)$
\emph{achieves $\theta$, }and that $\theta$ is \emph{feasible}.
This very natural and quite inviting problem was first proposed by
Aldous, Blanc and Curien in \cite{aldous_blanc_curien_distance_problem_2024},
where they also give some statistical motivation.

\subsection{Examples}

While any $\left(S,d,\mu\right)$ gives some distribution $\theta$
which can in principle be calculated, the inverse problem, of finding
an explicit space which achieves a given distribution, is harder.
The following examples show how to achieve some common distributions,
and may give some intuition about the types of constructions one can
use.
\begin{enumerate}
\item \textbf{The uniform measure}. Let $S=S^{1}=\r\mod\z$ be the unit
circle, $\mu$ the uniform measure on $S^{1}$, and $d\left(x,y\right)$
the arc length between $x$ and $y$. Then $d\left(X,Y\right)\sim U\left[0,\frac{1}{2}\right]$.
\item \textbf{The exponential distribution}. Let $S=\r$, let $d\left(x,y\right)=\abs{x-y}$,
and let $\mu=\mathrm{Exp}\left(1\right)$ be the exponential distribution.
By the memoryless property, $\abs{X-Y}$ is also exponentially distributed.
\item \textbf{\label{enu:bernoulli_construction}Bernoulli distribution}.
Let $p_{0}>0$. In order to achieve the Bernoulli distribution $\theta=p_{0}\delta_{0}+\left(1-p_{0}\right)\delta_{1}$,
let $m\geq2$ be an integer and let $\alpha\in\left[0,1\right]$ to
be chosen later. Let $S=\left\{ 0,\ldots,m-1\right\} $, equipped
with the discrete metric 
\[
d\left(x,y\right)=\begin{cases}
0 & x=y\\
1 & x\neq y.
\end{cases}
\]
Let $\mu$ be the distribution of the random variable on $S$ whose
value is $0$ with probability $\alpha$, and uniform among $\left\{ 1,\ldots,m-1\right\} $
with probability $1-\alpha$. Then 
\begin{equation}
\p\left[d\left(X,Y\right)=0\right]=\alpha^{2}+\left(\frac{1-\alpha}{m-1}\right)^{2}.\label{eq:bernoulli_probability}
\end{equation}
Choosing $m$ so that the quadratic equation $p_{0}=\alpha^{2}+\left(\frac{1-\alpha}{m-1}\right)^{2}$
has a solution and solving for $\alpha$ gives us the desired probability
space. 
\item \textbf{The uniform measure, revisited}. Let $S=\left\{ 0,1\right\} ^{\n}$,
let $\mu$ be the measure on $S$ where all entries are i.i.d. $0$-1
Bernoulli random variables with success probability $1/2$, and let
$d\left(x,y\right)=\sum_{k=1}^{\infty}2^{-k}\abs{x_{i}-y_{i}}$. Then
$d\left(X,Y\right)\sim U\left[0,1\right]$. This is essentially a
more explicit version of the first example.
\item \textbf{The exponential distribution, revisited}. Let $S=\left(0,1\right]$,
$\mu$ the uniform measure on $S$, and $d\left(x,y\right)=\abs{\int_{x}^{y}\frac{1}{t}dt}=\abs{\log\frac{y}{x}}$.
A short calculation shows that $\p\left[d\left(X,Y\right)\leq t\right]=1-e^{-t}$,
and so $d\left(X,Y\right)\sim\exp\left(1\right)$. 
\item \textbf{Non-example}. Let $\theta$ be any distribution with density
on $\nonnegative$, and denote its cumulative distribution function
(CDF) by $F$. Let $S=\nonnegative$, let 
\[
d\left(x,y\right)=\begin{cases}
0 & x=y\\
\max\left(x,y\right) & x\neq y,
\end{cases}
\]
and let $\mu$ be the distribution whose CDF is $\sqrt{F}$. Since
$\theta$ has density, the event $\left\{ x=y\right\} $ has measure
$0$ under $\mu$, and so $d\left(X,Y\right)=\max\left(X,Y\right)$
with probability $1$. The CDF of the maximum of two i.i.d. random
variables is the square of their individual CDF, and so $d\left(X,Y\right)\sim\theta$.
However, while $\mu$ is a Borel measure with respect to the Euclidean
topology, it is not a Borel measure with respect to the topology induced
by $d$. Note also that $\left(S,d\right)$ is not separable (for
small enough $\eps$, the only element $\eps$-close to $x$ is $x$
itself).
\end{enumerate}

\subsection{Results}

In \cite{aldous_blanc_curien_distance_problem_2024}, Aldous, Blanc
and Curien showed some basic constraints on the support of the distribution
-- for example, it must contain $0$ -- and gave several constructions
and approximations. In their main constructions, the metric space
$S$ was a weighted tree and the measure $\mu$ a distribution on
its leaves / rays. Choosing the weights and distribution in a clever
way, they showed that every finite discrete distribution with $0$
in its support is feasible. However, when applied to non-finite distributions
via finite approximations, their approach yields non-separable spaces.

Our first result is a new construction which does not involve taking
limits of finite distributions, avoiding limiting questions of separability.
This result resolves the question of discrete distributions.
\begin{thm}
\label{thm:discrete_distributions}Every nonnegative discrete distribution
$\theta$ supported on $0$ is feasible. \hyperlink{target:proof_of_discrete}{({proof $\nearrow$})}
\end{thm}

Rather than building a metric space $\left(S,d,\mu\right)$ which
directly achieves $\theta$ from scratch, our proof combines together
several simpler spaces $\left(S_{n},d_{n},\mu_{n}\right)$, allowing
us to sample from different $\theta_{n}$s with different probabilities.
The main challenge is to decompose $\theta$ into subdistributions
so that combining the spaces preserves the triangle inequality. As
will be described in the next section, this construction can also
be interpreted as a metric on the rays of a tree, where two rays are
close if they travel together for a long time, generalizing a commonly
used metric on the space of ends of a tree.

Our second result concerns distributions with density.
\begin{thm}
\label{thm:bounded_continuous_distributions}Let $g$ be a probability
density function on $\left[0,1\right]$, and suppose that there exist
constants $c,C>0$ such that 
\begin{equation}
c<g\left(t\right)<C\label{eq:bounds_on_density_function}
\end{equation}
for all $t\in\left[0,1\right]$. Then the distribution $\theta$ with
density $g$ is feasible. \hyperlink{target:proof_of_continuous}{({proof $\nearrow$})}
\end{thm}

The theorem is derived by altering the metric of a known-to-be-feasible
distribution; here, subadditivity is used in order to preserve the
metric. The theorem represents one step towards answering the open
problem in \cite{aldous_blanc_curien_distance_problem_2024}:
\begin{problem}
Prove that for every probability density function $f$ on $\nonnegative$
whose support contains $0$, the distribution with density $f$ is
feasible.
\end{problem}

\section{Proof of Theorem \ref{thm:discrete_distributions}}

\hypertarget{target:proof_of_discrete}Our proof uses the following
construction, which, given several feasible distributions, returns
a mixture of these distributions. The resultant distance can be thought
of as being computed by a ``selection'' mechanism: we linearly order
the distributions, and iteratively flip a coin for each one to see
whether or not it is the one that we sample from; if none of the distributions
are chosen, the default $0$ is returned. 

\paragraph{Selection space}

Denote the integers by $\z$ and the positive integers by $\n=\left\{ 1,2,\ldots\right\} $.
We say that a triplet $\left(S,d,\mu\right)$ is a measured metric
space if $\left(S,d\right)$ is a metric space and $\mu$ is a Borel
probability measure on $S$ with respect to $d$.

Let $I\subseteq\z$, and for each $n\in I$, let $\left(S_{n},d_{n},\mu_{n}\right)$
be a complete countable measured metric space with more than $1$
point. Assume that for every $n\in I\intersect\n$ there exists a
special $s_{n}\in S_{n}$ such that 
\begin{equation}
\sum_{n\in I\intersect\n}1-\mu_{n}\left(\left\{ s_{n}\right\} \right)<\infty.\label{eq:prob_sums_are_finite}
\end{equation}
We now define a new space $\left(S,d,\mu\right)$ out of these spaces;
later, in Proposition \ref{prop:selection_space_properties}, we give
conditions under which it is a complete separable measured metric
space. 
\begin{enumerate}
\item \textbf{The space}. We say that a vector $x\in\prod_{n\in I}S_{n}$
is \emph{eventually constant} if there exists an $N>0$ such that
$x_{n}=s_{n}$ for all $n\geq N$ (we still index the vector $x$
by the set $I$). Note that when $\sup I<\infty$, every vector is
eventually constant. We set 
\[
S=\left\{ x\in\prod_{n\in I}S_{n}\mid\text{\ensuremath{x} is eventually constant}\right\} .
\]
\item \textbf{The function }$d$. For $x\neq y\in S$, let $\index xy=\sup\left\{ n\mid x_{n}\neq y_{n}\right\} $.
This is always finite since $x$ and $y$ are eventually constant.
We then define 
\[
d\left(x,y\right)=\begin{cases}
0 & x=y\\
d_{\index xy}\left(x_{\index xy},y_{\index xy}\right) & x\neq y.
\end{cases}
\]
\item \textbf{The measure}. Let $\tilde{\mu}=\tensor_{n\in I}\mu_{n}$ be
the product measure on $\prod_{n\in I}S_{n}$. By (\ref{eq:prob_sums_are_finite})
the sum $\sum_{n\in I\intersect\n}1-\mu_{n}\left(\left\{ s_{n}\right\} \right)$
is finite, and the Borel-Cantelli lemma implies that $X\sim\tilde{\mu}$
is eventually constant with probability $1$ under $\tilde{\mu}$.
The measure $\tilde{\mu}$ can therefore be restricted to a measure
$\mu$ on $S$.
\end{enumerate}
We call the tuple $\left(S,d,\mu\right)$ the \emph{selection space}
of $\left\{ \left(S_{n},d_{n},\mu_{n}\right)\right\} _{n\in I}$.
Let us find the distribution of $d\left(X,Y\right)$ when $X,Y\sim\mu$
are independent. The probability of obtaining $0$ is  
\begin{equation}
\p\left[d\left(X,Y\right)=0\right]=\prod_{n\in I}\p\left[X_{n}=Y_{n}\right]=\prod_{n\in I}\p\left[d_{n}\left(X_{n},Y_{n}\right)=0\right],\label{eq:prob_of_0}
\end{equation}
while for $r>0$, 
\begin{align}
\p\left[d\left(X,Y\right)=r\right] & =\sum_{n\in I}\p\left[n\left(X,Y\right)=n\right]\cdot\p\left[d_{n}\left(X_{n},Y_{n}\right)=r\mid\index XY=n\right]\nonumber \\
 & =\sum_{n\in I}\left(\prod_{k\in I,k>n}\p\left[X_{k}=Y_{k}\right]\right)\p\left[X_{n}\neq Y_{n}\right]\cdot\p\left[d_{n}\left(X_{n},Y_{n}\right)=r\mid X_{n}\neq Y_{n}\right]\nonumber \\
 & =\sum_{n\in I}\left(\prod_{k\in I,k>n}\p\left[d\left(X_{k},Y_{k}\right)=0\right]\right)\p\left[d_{n}\left(X_{n},Y_{n}\right)=r\right],\label{eq:prob_of_positive}
\end{align}
and we have indeed obtained a mixture of positive elements of the
distributions achieved by $\left(S_{n},d_{n},\mu_{n}\right)$.
\begin{rem}
As a space of vectors indexed by $I\subseteq\z$, the selection space
can also be interpreted as the steps of a bi-directional walk, where
at time-step $n$, the walker may go in the ``direction'' $x_{n}\in S_{n}$.
When $\sup I$ is finite and we view time as starting at $\sup I$
and directed down towards $\inf I$, the index $\index xy$ is the
first time that two walkers $x$ and $y$ split up when starting at
a common origin. In this case the elements of $S$ can be seen as
the space of ends on some rooted tree, and the function $d$ is a
generalization of the metric $e^{-h\left(x,y\right)}$ on the space
of ends, where $h\left(x,y\right)$ is the depth of the last common
vertex of the rays $x$ and $y$. See Figure \ref{fig:space_of_ends_metric}.
When $\sup I=\infty$, however, there is no such common root. 

\begin{figure}[h]
\begin{centering}
\includegraphics[scale=0.75]{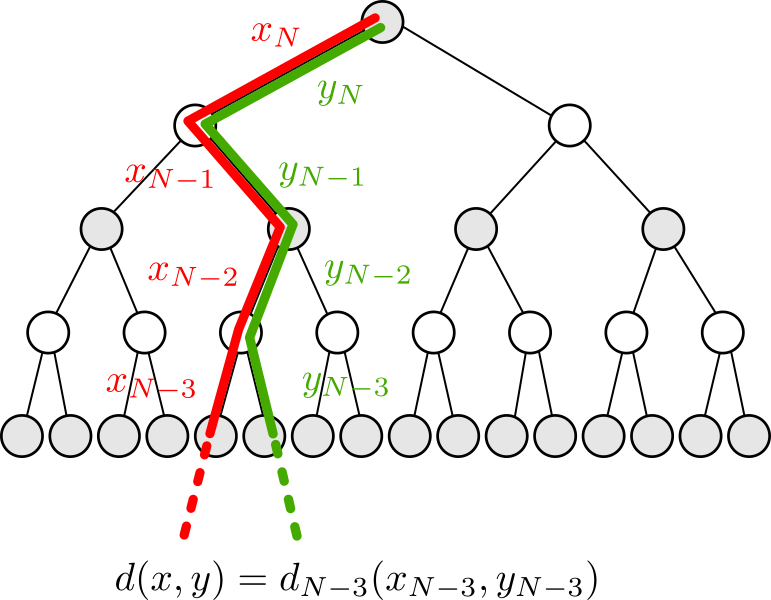}
\par\end{centering}
\caption{\label{fig:space_of_ends_metric}When $\sup I=N<\infty$, the function
$d$ is a generalization of a class of metrics on the space of ends
of a tree rooted, where the distance between two rays from the origin
depends on the depth of their last common vertex.}

\end{figure}
\end{rem}

\begin{prop}
\label{prop:selection_space_properties}Let $I\subseteq\z$ and $\left\{ \left(S_{n},d_{n},\mu_{n}\right)\right\} _{n\in I}$
be as above. Assume that the following two properties hold.
\begin{enumerate}
\item For every $m<n\in I$, every $x_{1}\neq x_{2}\in S_{m}$ and every
$y_{1}\neq y_{2}\in S_{n}$, 
\begin{equation}
d_{m}\left(x_{1},x_{2}\right)\leq2d_{n}\left(y_{1},y_{2}\right).\label{eq:metrics_are_subadditive_assumption}
\end{equation}
\item We have
\begin{equation}
\text{either }\inf I>-\infty\quad\text{ or }\quad\lim_{n\to-\infty}\sup_{z,w\in S_{n}}d_{n}\left(z,w\right)=0.\label{eq:assumption_separability}
\end{equation}
\end{enumerate}
Then the selection space $\left(S,d,\mu\right)$ is a complete separable
measured metric space. 
\end{prop}

\begin{proof}
\textbf{Metric}: It is clear that $d\left(x,y\right)=0$ if and only
if $x=y$, so we only need to verify the triangle inequality. Let
$x,y,z\in S$ be distinct, and denote $b=\index xy$.
\begin{enumerate}
\item If either $\index yz>b$ or $\index xz>b$, then since $x$ and $y$
agree on indices greater than $b$, we must have $\index xz=\index yz>b$,
and also $d\left(x,z\right)=d\left(y,z\right)$. Then by (\ref{eq:metrics_are_subadditive_assumption}),
\[
d\left(x,y\right)=d_{b}\left(x_{b},y_{b}\right)\leq2d_{\index xz}\left(x_{\index xz},z_{\index xz}\right)=d\left(x,z\right)+d\left(y,z\right).
\]
\item If both $n\left(y,z\right)=b$ and $n\left(x,z\right)=b$, then the
triangle inequality is satisfied because $d_{b}$ is a metric.
\item Otherwise, we have $n\left(y,z\right),n\left(x,z\right)\leq b$, and
at least one of $n\left(y,z\right)$ or $n\left(x,z\right)$ is strictly
smaller than $b.$ But it is impossible for both to be strictly smaller,
since that would mean that $x_{b}=z_{b}=y_{b}$, contradicting the
fact that $x_{b}\neq y_{b}$. Supposing w.l.o.g. that $\index xz=b$,
we then trivially have
\[
d\left(x,y\right)=d_{b}\left(x_{b},y_{b}\right)=d_{b}\left(x_{b},z_{b}\right)=d\left(x,z\right)\leq d\left(x,z\right)+d\left(y,z\right).
\]
\end{enumerate}
\textbf{Completeness}: Let $N\in I$, and let $a\neq b\in S_{N}$.
Then by (\ref{eq:metrics_are_subadditive_assumption}), for every
$n>N$ and every $w,z\in S_{n}$, $d_{n}\left(w,z\right)\geq\frac{1}{2}d_{N}\left(a,b\right)$.
Thus, for every $N\in I$,
\begin{equation}
\inf_{n>N}\inf_{w,z\in S_{n}}d_{n}\left(w,z\right)>0.\label{eq:inf_on_tail}
\end{equation}
Let $\left(x^{\left(k\right)}\right)_{k\in\n}$ be a Cauchy sequence
in $S$, and let $L=\inf I$. When $L=-\infty$, (\ref{eq:inf_on_tail})
implies that if $d\left(x^{\left(k_{1}\right)},x^{\left(k_{2}\right)}\right)<\eps$
for some $k_{1}$ and all $k_{2}\geq k_{1}$ and small enough $\eps$,
then $x_{n}^{\left(k_{2}\right)}=x_{n}^{\left(k_{1}\right)}$ for
all $n>N$ for some $N\in I$, with $N\to-\infty$ as $\eps\to0$.
Thus $x^{\left(k\right)}$ converges to an element in $S$. When $L>-\infty$,
(\ref{eq:inf_on_tail}) implies that for small enough $\eps$, $x_{n}^{\left(k_{2}\right)}=x_{n}^{\left(k_{1}\right)}$
for all $n>L$, while the first elements $\left(x_{L}^{\left(k\right)}\right)_{k\in\n}$
form a Cauchy sequence in $S_{L}$. Since $S_{L}$ is itself complete,
we again get that $x^{\left(k\right)}$ converges to an element in
$S$. 

\textbf{Separability}: If $\inf I>-\infty$, then $S$ is a countable
union of countable sets, and so is countable itself. Otherwise, by
(\ref{eq:assumption_separability}), $\lim_{n\to-\infty}\sup_{z,w\in S_{n}}d\left(z,w\right)=0.$
For each negative $n\in I$, let $a^{\left(n\right)}\in S_{n}$, and
set 
\[
A_{n}=\left\{ x\in S\mid x_{i}=a^{\left(i\right)}\,\all i\leq n\right\} .
\]
Since all vectors in $S$ are eventually constant, each $A_{n}$ is
a countable union of countable sets and is therefore countable, and
so the union $\union_{n\in I}A$ is also countable. It is also dense
in $S$: for any $x\in S$ and any $n\in I$, the set $A_{n}$ contains
an element $y$ such that $y_{k}=x_{k}$ for all $k>n$, and so $n\left(x,y\right)\leq n$;
denseness follows since $\sup_{z,w\in S_{n}}d_{n}\left(z,w\right)\to0$
as $n\to-\infty$.

\textbf{Borel measure}: Under the metric $d$, an open ball around
$x\in S$ of radius $r$ is given by 
\[
B\left(x,r\right)=\bigcup_{N\in I}\,\,\,\bigcup_{s\in S_{N}\mid d_{N}\left(x_{N},s\right)<r,s\neq x_{N}}\left\{ y\in S\mid\text{\ensuremath{y_{N}=s},\ensuremath{y_{n}=x_{n}} for all \ensuremath{n>N}}\right\} \union\left\{ x\right\} .
\]
Each set of the form $\left\{ y\in S\mid\text{\ensuremath{y_{N}=s},\ensuremath{y_{n}=x_{n}} for all \ensuremath{n>N}}\right\} $
is measurable under the product measure $\tilde{\mu}$, and therefore
so is the countable union of such sets. The restriction measure $\mu$
is therefore defined on the $\sigma$-algebra generated by the open
balls in $S$. Since $\left(S,d\right)$ is separable, $\mu$ is Borel.
\end{proof}
\begin{prop}
\label{prop:positive_interval_is_feasible}Let $a>0$, let $p_{0}\in\left(0,1\right)$
and let $\lambda$ be a discrete probability distribution with $\mathrm{supp}\left(\lambda\right)\subseteq$$\left[a,2a\right]$.
Then the distribution $\theta=p_{0}\delta_{0}+\left(1-p_{0}\right)\lambda$
is feasible. 
\end{prop}

\begin{proof}
Let $a_{0}=0$, and let $\left(a_{n}\right)_{n\in I}$ be the positive
atoms of $\lambda$ arranged in some arbitrary order; here, $I=\n$
if there is an infinite number of atoms and $I=\left[N\right]$ for
some $N$ otherwise. For every $n\in I$, denote $p_{n}=\theta\left(\left\{ a_{n}\right\} \right)$,
and define $q_{n}$ by
\[
q_{n}=\frac{p_{n}}{p_{0}+p_{1}+\ldots+p_{n}}.
\]
For $n\in I$, let $\left(S_{n},d_{n},\mu_{n}\right)$ be the metric
space corresponding to the Bernoulli measure $\theta_{n}=\left(1-q_{n}\right)\delta_{0}+q_{n}\delta_{a_{n}}$,
obtained by multiplying the metric described in the examples section
by $a_{n}$. Then the assumption in the construction of the selection
space and the two conditions in Proposition \ref{prop:selection_space_properties}
hold for $\left\{ \left(S_{n},d_{n},\mu_{n}\right)\right\} _{n\in I}$:
\begin{enumerate}
\item When constructing a Bernoulli random variable $p\delta_{0}+\left(1-p\right)\delta_{1}$with
probability $p$ of obtaining $0$ as in example \ref{enu:bernoulli_construction},
the parameters $m$ and $\alpha$ may be chosen so that $\alpha\geq p$.
Indeed, setting $f_{m}\left(x\right)=x^{2}+\left(\frac{1-x}{m-1}\right)^{2}$,
we have $p=f_{m}\left(\alpha\right)$. Since $\lim_{m\to\infty}f_{m}\left(p\right)=p^{2}<p$,
there exists an $m$ large enough so that $f_{m}\left(p\right)<p$,
and by continuity of $f_{m}$ and the fact that $f_{m}\left(1\right)=1$,
there exists $\alpha\in\left[p,1\right]$ with $f_{m}\left(\alpha\right)=p$
as required. This shows that the probability to obtain $0$ under
$\mu_{n}$ is bounded by
\begin{equation}
\mu_{n}\left(\left\{ 0\right\} \right)\geq1-q_{n}.\label{eq:bound_on_q}
\end{equation}
Thus 
\begin{equation}
\sum_{n\in I}1-\mu_{n}\left(\left\{ 0\right\} \right)\leq\sum_{n\in I}q_{n}=\sum_{n\in I}\frac{p_{n}}{p_{0}+p_{1}+\ldots+p_{n}}\leq\sum_{n\in I}\frac{p_{n}}{p_{0}}=\frac{1-p_{0}}{p_{0}},\label{eq:sum_of_q_is_finite}
\end{equation}
and the sum in (\ref{eq:prob_sums_are_finite}) is finite. 
\item Since $a_{n}\in\left[a,2a\right]$, the requirement (\ref{eq:metrics_are_subadditive_assumption})
is immediate: the left-hand side of (\ref{eq:metrics_are_subadditive_assumption})
is at most $2a$, while the right-hand side is at least $2a$. 
\item $\inf I=1>-\infty$ as required in (\ref{eq:assumption_separability}).
\end{enumerate}
Thus, the selection space $\left(S,d,\mu\right)$ constructed from
$\left\{ \left(S_{n},d_{n},\mu_{n}\right)\right\} _{n\in I}$ is a
complete separable measured metric space. Since the positive atoms
of the $\mu_{n}$s are all distinct, equations (\ref{eq:prob_of_0})
and (\ref{eq:prob_of_positive}) imply that 
\begin{equation}
\p\left[d\left(X,Y\right)=0\right]=\prod_{n\in I}\p\left[d_{n}\left(X_{n},Y_{n}\right)=0\right]=\prod_{n\in I}\left(1-q_{n}\right)=p_{0},\label{eq:single_prob_zero}
\end{equation}
and
\[
\p\left[d\left(X,Y\right)=a_{n}\right]=q_{n}\prod_{i\in I,i>n}\left(1-q_{i}\right)=p_{n}
\]
as needed.

We remark that by (\ref{eq:bound_on_q}) and (\ref{eq:single_prob_zero}),
\begin{equation}
\p\left[X=\left(0,0,\ldots\right)\right]=\prod_{n\in I}\mu_{n}\left(\left\{ 0\right\} \right)\geq\prod_{n\in I}\left(1-q_{n}\right)=p_{0}.\label{eq:entire_vector_is_0_whp}
\end{equation}
\end{proof}
\begin{proof}[Proof of Theorem \ref{thm:discrete_distributions}]
The measure $\theta$ can be written as 
\[
\theta=p_{\infty}\delta_{0}+\sum_{n\in I}p_{n}\lambda_{n},
\]
where $I\subseteq\z$, and for every $n\in I$, $p_{n}>0$ and the
measure $\lambda_{n}$ is discrete and supported on $\left(2^{n},2^{n+1}\right]$.
Note that if $p_{\infty}=0$, then necessarily $\inf I=-\infty$.
Define the numbers 
\[
\beta_{n}=\frac{p_{n}}{1-\sum_{i\in I,i>n}p_{i}},
\]
and let $\left(S_{n},d_{n},\mu_{n}\right)$ be the metric space defined
in the proof of Proposition \ref{prop:positive_interval_is_feasible}
corresponding to the measure $\theta_{n}=\left(1-\beta_{n}\right)\delta_{0}+\beta_{n}\lambda_{n}$.
Note that each $S_{n}$ is countable. Then the assumption in the construction
of the selection space and the two conditions in Proposition \ref{prop:selection_space_properties}
hold for $\left\{ \left(S_{n},d_{n},\mu_{n}\right)\right\} _{n\in I}$:
\begin{enumerate}
\item If $\sup I<\infty$ then of course the sum in (\ref{eq:prob_sums_are_finite})
is finite. Otherwise, by (\ref{eq:entire_vector_is_0_whp}), if $X_{n}\sim\mu_{n}$,
then 
\[
\p\left[X_{n}=\left(0,0,\ldots\right)\right]\geq1-\beta_{n}.
\]
Thus
\[
\sum_{n\in I\intersect\n}\p\left[X_{n}\neq\left(0,0,\ldots\right)\right]\leq\sum_{n\in I\intersect\n}1-\left(1-\beta_{n}\right)=\sum_{n\in I\intersect\n}\beta_{n}\leq\sum_{n\in I\intersect\n}\frac{p_{n}}{1-\sum_{i\in I,i>n}p_{i}}\leq\frac{1}{p_{\infty}+\sum_{i\in I,i<1}p_{i}}.
\]
\item Since the range of $d_{n}$ is $\left(2^{n},2^{n+1}\right]$, the
requirement (\ref{eq:metrics_are_subadditive_assumption}) is immediate;
in fact, the left hand side of (\ref{eq:metrics_are_subadditive_assumption})
is bounded by just half of the right hand side.
\item $\lim_{n\to-\infty}\sup_{z,w\in S_{n}}d_{n}\left(z,w\right)=\lim_{n\to-\infty}2^{n+1}=0$
as required in (\ref{eq:assumption_separability}).
\end{enumerate}
Thus, the selection space $\left(S,d,\mu\right)$ constructed from
$\left\{ \left(S_{n},d_{n},\mu_{n}\right)\right\} _{n\in I}$ is a
complete separable measured metric space. Since the supports of $\lambda_{n}$
are disjoint, equations (\ref{eq:prob_of_0}) and (\ref{eq:prob_of_positive})
imply that

\[
\p\left[d\left(X,Y\right)=0\right]=\prod_{n\in I}\p\left[X_{n}=Y_{n}\right]=\prod_{n\in I}\left(1-\beta_{n}\right)=p_{\infty},
\]
while the probability to sample from $\lambda_{n}$ is exactly $\beta_{n}\prod_{i\in I,i>n}\left(1-\beta_{i}\right)=p_{n}$,
as needed. 
\end{proof}

\section{Proof of Theorem \ref{thm:bounded_continuous_distributions}}

\hypertarget{target:proof_of_continuous}Our proof first finds an
easy-to-construct feasible ``starter'' distribution, and then changes
its metric in order to obtain the distribution $\theta$. The idea
is as follows. Let $W$ be a random variable on $\left[0,1\right]$
whose cumulative distribution function $F_{W}$ is continuous and
strictly increasing on $\left[0,1\right]$, and let $\varphi:\left[0,1\right]\to\left[0,1\right]$
be a strictly increasing bijection. The cumulative distribution function
of the random variable $\varphi\left(W\right)$ is given by
\[
F_{\varphi\left(W\right)}\left(t\right)=\p\left[\varphi\left(W\right)\leq t\right]=\p\left[W\leq\varphi^{-1}\left(t\right)\right]=F_{W}\left(\varphi^{-1}\left(t\right)\right).
\]
Denoting by $G$ the cumulative distribution function of $\theta$
and choosing $\varphi\left(t\right)=G^{-1}\left(F_{W}\left(t\right)\right)$
gives us, for $t\in\left[0,1\right]$, $F_{W}\left(\varphi^{-1}\left(t\right)\right)=G\left(t\right)$,
and so the random variable $\varphi\left(W\right)$ has distribution
$\theta$. If $W$ is achievable by some $\left(S,d,\mu\right)$,
then when $X,Y\sim\mu$ are independent, we have $\varphi\left(d\left(X,Y\right)\right)\sim\theta$.
However, this does not immediately mean that $\left(S,\varphi\circ d,\mu\right)$
is a measured metric space that achieves $\theta$, since there is
no guarantee that $\varphi\circ d$ is still a metric (or that $\mu$
is still Borel under the topology induced by $\varphi\circ d$). The
essence of the proof is to find a suitable $W$ which allows $\varphi$
to preserve the metric $d$.
\begin{defn}
A function $f:\nonnegative\to\nonnegative$ is called \emph{metric
preserving }if for every metric space $\left(S,d\right)$, the function
$f\circ d$ is also a metric on $S$. It is called \emph{strongly
metric preserving} if it is metric preserving and $\left(S,d\right)$
is topologically equivalent to $\left(S,f\circ d\right)$.
\end{defn}

Metric preserving functions have been well studied; see \cite{corazza_introduction_to_metric_preserving_functions}
for a brief introduction. We will require only the following standard
result.
\begin{thm}
\label{thm:subadditivity_implies_metric_preserving}\cite[page 131]{kelley_general_topology}
Let $f:\nonnegative\to\nonnegative$ be continuous, nondecreasing,
and satisfying the following two conditions:
\begin{enumerate}
\item $f\left(x\right)=0\iff x=0$.
\item (Subadditivity) $f\left(x+y\right)\leq f\left(x\right)+f\left(y\right)$
for all $x,y\in\nonnegative$.
\end{enumerate}
Then $f$ is strongly metric preserving.
\end{thm}

Theorem \ref{thm:subadditivity_implies_metric_preserving} implies
that it suffices to find a continuous random variable $W$ such that
$G^{-1}\left(F_{W}\left(t\right)\right)$ is subadditive. If $G^{-1}$
is itself already subadditive, then no modification by $F_{W}$ is
actually needed; in this case, we can take $W$ to be the uniform
measure on $\left[0,1\right]$, so that $F_{W}\left(t\right)=t$ for
$t\in\left[0,1\right]$ and $G^{-1}\left(F_{W}\left(t\right)\right)=G^{-1}\left(t\right)$.
Obtaining $\theta$ is then just an instance of the inverse sampling
theorem (as we saw in the examples section, the uniform measure is
indeed feasible). However, when $G^{-1}$ is not subadditive, the
idea is that if $F_{W}$ itself is in some sense very strongly subadditive,
then its subadditivity is enough to counter the non-subadditivity
of $G^{-1}$. This is made precise in the following simple proposition. 
\begin{prop}
\label{prop:composing_into_subadditivity}Let $\Psi:\left[0,1\right]\to\left[0,1\right]$
be an absolutely continuous nondecreasing function, and assume that
there exist $m,M>0$ such that
\begin{equation}
m<\Psi'\left(t\right)<M.\label{eq:h_has_pinched_derivatives}
\end{equation}
Let $f:\nonnegative\to\left[0,1\right]$ be such that for every $x,y\in\nonnegative$
with $x\leq y$, 
\begin{equation}
f\left(x+y\right)\leq\frac{m}{M}f\left(x\right)+f\left(y\right).\label{eq:f_is_strongly_subadditive}
\end{equation}
Then $\Psi\circ f$ is subadditive.
\end{prop}

\begin{proof}
By absolute continuity, for all $x\leq y$,
\begin{align*}
\Psi\left(f\left(x+y\right)\right) & =\Psi\left(f\left(y\right)\right)+\int_{f\left(y\right)}^{f\left(x+y\right)}\Psi'\left(t\right)\,dt\\
 & \stackrel{\eqref{eq:h_has_pinched_derivatives}}{\leq}\Psi\left(f\left(y\right)\right)+M\left(f\left(x+y\right)-f\left(y\right)\right)\\
 & \stackrel{\eqref{eq:f_is_strongly_subadditive}}{\leq}\Psi\left(f\left(y\right)\right)+mf\left(x\right)\\
 & \stackrel{\eqref{eq:h_has_pinched_derivatives}}{\leq}\Psi\left(f\left(y\right)\right)+\Psi\left(f\left(x\right)\right).
\end{align*}
\end{proof}
The next proposition shows that it is indeed possible to obtain such
$f$. 
\begin{prop}
\label{prop:strongly_subadditive_functions_are_realizable}For every
$\eps\in\left(0,1\right)$, there exists a feasible continuous random
variable $W$ whose cumulative distribution function $F_{W}$ satisfies
\[
F_{W}\left(x+y\right)\le\eps F_{W}\left(x\right)+F_{W}\left(y\right)
\]
for all $x\leq y$. 
\end{prop}

Given the above, Theorem \ref{thm:bounded_continuous_distributions}
quickly follows.
\begin{proof}[Proof of Theorem \ref{thm:bounded_continuous_distributions}]
By (\ref{eq:bounds_on_density_function}), $G^{-1}$ is absolutely
continuous, strictly increasing, and has upper and lower bounds on
its derivative $\left(G^{-1}\right)\left(t\right)'=\frac{1}{g\left(G^{-1}\left(t\right)\right)}$.
Then by Propositions \ref{prop:composing_into_subadditivity} and
\ref{prop:strongly_subadditive_functions_are_realizable}, there is
a feasible $W$ achieved by $\left(S,d,\mu\right)$ such that $\varphi=G^{-1}\circ F_{W}$
is subadditive. Since both $G^{-1}$ and $F_{W}$ are continuous,
by Theorem \ref{thm:subadditivity_implies_metric_preserving}, $\varphi$
is strongly metric preserving, and so $\left(S,\varphi\circ d,\mu\right)$
achieves $\theta$. 
\end{proof}
\begin{proof}[Proof of Proposition \ref{prop:strongly_subadditive_functions_are_realizable}]
Let $\alpha\in\left(0,\frac{1}{2}\right)$ to be chosen later, define
$H:\r\to\left[0,1\right]$ by 
\[
H\left(t\right)=\begin{cases}
0 & t<0\\
t^{\alpha} & 0\leq t\leq1\\
1 & t>1,
\end{cases}
\]
and let $\mu$ be the distribution whose cumulative distribution function
is $H$. This distribution has density $h\left(t\right)=\alpha t^{\alpha-1}$
for $t\in\left[0,1\right]$. Let $W=\abs{X-Y}$, where $X,Y\sim\mu$
are i.i.d. By definition, $W$ is feasible. For $t\in\left[0,1\right]$,
the cumulative distribution function $F_{W}$ is given by 
\begin{align*}
F_{W}\left(t\right) & =\int_{-\infty}^{\infty}\int_{-\infty}^{\infty}h\left(x\right)h\left(y\right)\one_{\abs{x-y}\leq t}\,dydx\\
 & =2\int_{0}^{1}\int_{x}^{x+t}h\left(x\right)h\left(y\right)\,dydx\\
 & =2\int_{0}^{1}h\left(x\right)\left(H\left(x+t\right)-H\left(x\right)\right)\,dx\\
 & =2\int_{0}^{1-t}\alpha x^{2\alpha-1}\left(\left(1+\frac{t}{x}\right)^{\alpha}-1\right)\,dx+2\int_{1-t}^{1}\alpha x^{\alpha-1}\left(1-x^{\alpha}\right)\,dx.
\end{align*}
The function $F_{W}$ can readily be shown to be concave on $\nonnegative$
by calculating the second derivative of $2\int_{0}^{1}\int_{x}^{x+t}h\left(x\right)h\left(y\right)dydx$.
Thus, for all $z,w\in\nonnegative$, 
\[
F_{W}\left(z\right)\leq F_{W}\left(w\right)+\left(z-w\right)F_{W}'\left(w\right).
\]
Let $x,y\in\left(0,\infty\right)$ with $x\leq y$. Choosing $z=x+y$
and $w=y$, we get 
\begin{align*}
F_{W}\left(x+y\right) & \leq F_{W}\left(y\right)+xF_{W}'\left(y\right)\\
 & \leq F_{W}\left(y\right)+xF_{W}'\left(x\right)\\
 & =F_{W}\left(y\right)+F_{W}\left(x\right)\cdot\frac{xF_{W}'\left(x\right)}{F_{W}\left(x\right)}.
\end{align*}
It therefore suffices to show that for all $t\in\nonnegative$,
\[
\frac{tF_{W}'\left(t\right)}{F_{W}\left(t\right)}\leq\eps.
\]
As $F_{W}$ is concave and increasing, it is enough to show a bound
of $\frac{1}{2}\eps$ for all $t\leq\frac{1}{2}$, since for all $t>\frac{1}{2}$,
$F_{W}'\left(t\right)\leq F_{W}'\left(\frac{1}{2}\right)$ and $F_{W}\left(t\right)\geq F_{W}\left(\frac{1}{2}\right)$.
Assuming that $t\leq\frac{1}{2}$, the denominator is bounded by

\begin{align*}
F_{W}\left(t\right) & =2\alpha\left(\int_{0}^{1-t}x^{2\alpha-1}\left(\left(1+\frac{t}{x}\right)^{\alpha}-1\right)\,dx+\int_{1-t}^{1}x^{\alpha-1}\left(1-x^{\alpha}\right)\,dx\right)\\
 & \geq2\alpha\int_{0}^{t}x^{2\alpha-1}\left(\left(\frac{t}{x}\right)^{\alpha}-1\right)\,dx=2\alpha\left(\frac{1}{\alpha}t^{2\alpha}-\frac{1}{2\alpha}t^{2\alpha}\right)=t^{2\alpha}.
\end{align*}
Recalling that $\alpha<\frac{1}{2}$, the numerator is bounded by 

\begin{align*}
F_{W}'\left(t\right) & =2\int_{0}^{1}h\left(x\right)h\left(x+t\right)\,dx\\
 & =2\alpha^{2}\left(\int_{0}^{t}x^{2\alpha-2}\left(1+\frac{t}{x}\right)^{\alpha-1}\,dx+\int_{t}^{1-t}x^{2\alpha-2}\left(1+\frac{t}{x}\right)^{\alpha-1}\,dx\right)\\
 & \leq2\alpha^{2}\left(\int_{0}^{t}x^{2\alpha-2}\left(\frac{t}{x}\right)^{\alpha-1}\,dx+\int_{t}^{1-t}x^{2\alpha-2}\,dx\right)\\
 & =2\alpha t^{2\alpha-1}-2\alpha^{2}\frac{1}{1-2\alpha}\left(\frac{1}{\left(1-t\right)^{1-2\alpha}}-\frac{1}{t^{1-2\alpha}}\right)\\
 & \leq2\alpha\frac{1}{t^{1-2\alpha}}\left(1+\frac{\alpha}{1-2\alpha}\right).
\end{align*}
Choosing $\alpha\leq\frac{1}{4}$ then gives $F_{W}'\left(t\right)\leq4\alpha t^{2\alpha-1}$.
We thus have 
\[
\frac{tF_{W}'\left(t\right)}{F_{W}\left(t\right)}\leq\frac{4\alpha t^{2\alpha}}{t^{2\alpha}}=4\alpha,
\]
and the proposition follows for $\alpha\leq\frac{1}{8}\eps$.
\end{proof}
\begin{rem}
The function $F_{W}$ may be expressed in a slightly more closed form
using hypergeometric functions; a short calculation reveals that for
$t\in\left[0,1\right]$, 
\[
F_{W}\left(t\right)=1+2\left(1-t\right)^{\alpha}\left(\hypergeometric{-\alpha}1{1+\alpha}{1-t}-1\right),
\]
where $\hypergeometric abcz$ is the hypergeometric function
\[
\hypergeometric abcz=\frac{\Gamma\left(c\right)}{\Gamma\left(b\right)\Gamma\left(c-b\right)}\int_{0}^{1}\frac{x^{b-1}\left(1-x\right)^{c-b-1}}{\left(1-xz\right)^{a}}\,dx.
\]
A plot of $F_{W}$ for some values of $\alpha$ is given in Figure
\ref{fig:F_W}, where the asymptotic behavior of $t^{2\alpha}$ can
be seen near the origin.

\begin{figure}[h]
\begin{centering}
\includegraphics[scale=0.5]{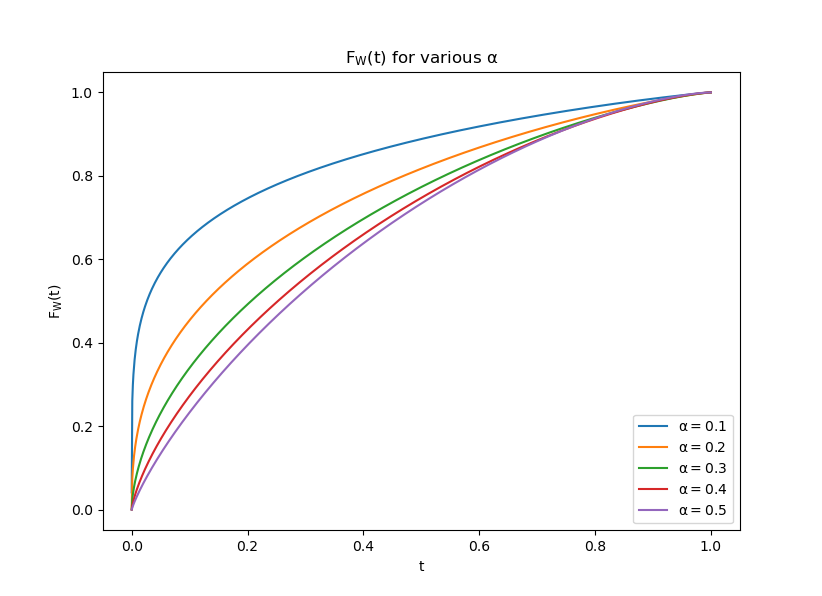}
\par\end{centering}
\caption{The function $F_{W}$ for various values of $\alpha$.\label{fig:F_W}}

\end{figure}
\end{rem}

\begin{rem}
It might perhaps be possible to extend the proof method of Theorem
\ref{thm:bounded_continuous_distributions} to distributions on $\nonnegative$
where the density is always positive. Indeed, if we compose $G^{-1}$
with another function $f$, we still have 
\begin{align*}
G^{-1}\left(f\left(x+y\right)\right) & =G^{-1}\left(f\left(y\right)\right)+\int_{f\left(y\right)}^{f\left(x+y\right)}\frac{d}{dt}G^{-1}\left(t\right)\,dt\\
 & \leq G^{-1}\left(f\left(y\right)\right)+\left(f\left(x+y\right)-f\left(y\right)\right)\sup_{t\in\left[f\left(y\right),f\left(x+y\right)\right]}\frac{d}{dt}G^{-1}\left(t\right).
\end{align*}
An analogue of Theorem \ref{thm:bounded_continuous_distributions}
would follow if we could find, for each $\eps>0$, an unbounded feasible
random variable $W$ such that for all $x\leq y$
\[
F_{W}\left(x+y\right)\leq F_{W}\left(y\right)+\left(\eps\sup_{t\in\left[F_{W}\left(y\right),F_{W}\left(x+y\right)\right]}\frac{d}{dt}G^{-1}\left(t\right)\right)F_{W}\left(x\right).
\]
Unlike the random variable from Proposition \ref{prop:strongly_subadditive_functions_are_realizable},
such a $W$ would have to rely directly on $G$ in some way. However,
the method cannot be easily extended to distributions whose density
is $0$ in an interval: this would imply that $G^{-1}\circ F_{W}$
is discontinuous, which would prevent it from preserving the metric. 
\end{rem}

\section{Acknowledgments}

We thank Wendelin Werner and Guillaume Blanc for useful comments,
and the anonymous referee for their corrections and suggestions.

\bibliographystyle{alpha}
\bibliography{distance_problem}

\end{document}